\documentclass[12pt]{amsart}
\usepackage{graphicx} 
\usepackage[utf8]{inputenc}
\usepackage{geometry}
\usepackage{amssymb}
\usepackage{indentfirst,amscd}
\usepackage{amsthm, amsfonts, empheq, amsmath}
\usepackage{mathrsfs}
\usepackage{mathtools}
\usepackage{enumitem}
\usepackage{tikz-cd} 
\usepackage{adjustbox}
\usepackage[style=alphabetic, giveninits=true, maxbibnames=99]{biblatex}
\renewbibmacro{in:}{}
\usepackage{thmtools}
\usepackage{thm-restate}
\usepackage{hyperref}
\usepackage{faktor}

\makeatletter
\newcommand*\bigcdot{\mathpalette\bigcdot@{.55}}
\newcommand*\bigcdot@[2]{\mathbin{\vcenter{\hbox{\scalebox{#2}{$\m@th#1 \bullet$}}}}}
\makeatother

\theoremstyle{definition}
\declaretheorem[name=Definition, numberwithin=section]{Def}
\newtheorem*{Claim}{Claim}

\theoremstyle{plain}
\declaretheorem[name=Theorem, sibling=Def]{Theo}
\declaretheorem[name=Lemma, sibling=Def]{Lemma}
\declaretheorem[name=Proposition, sibling=Def]{Prop}
\declaretheorem[name=Corollary, sibling=Def]{Cor}

\theoremstyle{remark}
\declaretheorem[name=Remark, sibling=Def]{Rem}

\newcommand{\g}{\mathfrak{g}}
\newcommand{\h}{\mathfrak{h}}

\DeclareMathOperator{\Alb}{Alb}

\DeclareMathOperator{\Tr}{Tr}

\DeclareMathOperator{\Ric}{Ric}

\DeclareMathOperator{\Id}{Id}
\DeclareMathOperator{\Iso}{Iso}

\DeclareMathOperator{\End}{End}
\DeclareMathOperator{\rk}{rk}
\DeclareMathOperator{\Hom}{Hom}
\DeclareMathOperator{\Span}{span}
\DeclareMathOperator{\codim}{codim}

\newcommand{\X}{\mathcal{X}}
\newcommand{\J}{\bar{J}}
\newcommand{\A}{\mathcal{A}}
\newcommand{\HH}{\mathcal{H}}
\newcommand{\Lie}{\mathscr{L}}
\newcommand{\F}{\mathcal{F}}
\newcommand{\C}{\mathbb{C}}
\newcommand{\Z}{\mathbb{Z}}
\newcommand{\R}{\mathbb{R}}

\title[On the Classification of V. Manifolds with $c_{1, B} = 0$ and large $b_1$]{On the Classification of Vaisman Manifolds with Vanishing First Basic Chern Class and Large First Betti Number}
\author[Lucas Gomes]{Lucas H. S. Gomes}
\address{Department of Mathematics, Graduate School of Science, The University of Osaka, Osaka, Japan}
\email{u788855c@ecs.osaka-u.ac.jp}
\begin{document}
	\begin{abstract}
		We show that every Vaisman manifold with large first Betti number and vanishing first basic Chern class is diffeomorphic to a Kodaira-Thurston manifold. Furthermore, its complex structure is left-invariant, the characteristic foliation is regular, and the associated fibration is given by the Albanese map. Under the additional assumption that the LCK rank is $1$, the Vaisman structure is also left-invariant. We further prove that if all basic harmonic $1$-forms have constant length, then the Vaisman manifold with large first Betti number is diffeomorphic to a Kodaira-Thurston manifold and its complex structure is the standard complex structure. Finally,  we discuss the relationship of this condition with transverse geometric formality in this setting.
	\end{abstract}

	\maketitle
	\tableofcontents
   	\section{Introduction}
	
A Locally Conformally K\"ahler manifold is a Hermitian manifold $(M, J, \omega)$ together with a closed $1$-form $\theta$ such that $d\omega = \theta \wedge \omega$. When $M$ is compact and $\theta$ is parallel, we call $M$ a Vaisman manifold. They posses a natural foliation structure with a transverse K\"ahler metric. In K\"ahler geometry the class of Calabi-Yau manifolds is defined as the K\"ahler manifolds with vanishing first Chern class $c_1$. It is an important class of manifolds both for mathematics and mathematical physics, playing an essential part for string theory, mirror symmetry and other applications. However, for non-K\"ahler geometry such as LCK metrics, the condition $c_1(M) = 0$ might not be as useful and one might use more suitable cohomology theories to obtain a more close analogous results from K\"ahler geometry. For example, Istrati studied Vaisman manifolds with vanishing first Chern-Bott class $c_1^{BC}$ in \cite{istrati2023vaismanmanifoldsvanishingchern} obtaining a Vaisman version of the Beauville-Bogomolov decomposition theorem. In \cite{barbaro2025calabiyaulocallyconformallykahler} the authors showed that left-invariant LCK metrics with $c_1^{BC} = 0$ are Vaisman in the solvmanifold setting. In this work, we concern ourselves with the transverse version of this condition. The natural transverse K\"ahler metric of a Vaisman manifold allows us to define the basic Chern classes $c_{i, B}$, which are introduced in the preliminaries section. When the foliation is regular and quasi-regular, they descend naturally to the Chern classes of the quotient space. With a further topological restriction we obtain the following classification theorem, which is the main theorem of this work.

\begin{Theo}[Theorem \ref{MainTheorem}]
	Let $(M^{2n+2}, J, \omega, \theta)$ be a Vaisman manifold with $b_1(M) = 2n+1$ and $c_{1, B}(M) = 0$. Then, $M$ is diffeomorphic to a Kodaira-Thurston manifold and $J$ is left-invariant. Moreover, the characteristic foliation $\Sigma$ is regular, $M$ is a fiber bundle over the Albanese torus
	$$ T^2 \longrightarrow M \xlongrightarrow{\alpha} \Alb (M)$$
	where $\alpha$ is the Albanese map, whose fibers are the leaves of $\Sigma$, which are biholomorphic to a complex torus $T^2$. 
\end{Theo}
When $c_{1, B}(M) = 0$, $b_1(M) = 2n+1$ is the highest value we can impose on a Vaisman manifold (see \ref{Rem:MainTheo}).
The usual Beauville-Bogomolov decomposition theorem (\cite{Bogomolov,Beauville}) in K\"ahler geometry says that a Calabi-Yau manifold $M$ up to a finite covering splits as a product of a complex torus and a simply-connected manifold. If one imposes $b_1(M) = \dim_\R M$, then it follows that $M$ is biholomorphic to a complex torus. Thus, our main theorem is the analogous version of this result for Vaisman geometry. We emphasize that our theorem and its proof does not rely on the work of Istrati in \cite{istrati2023vaismanmanifoldsvanishingchern}. Compare with \cite[Corollary 5.11]{istrati2023vaismanmanifoldsvanishingchern}. In section 8 we further refine this above theorem for the case of Vaisman manifolds with LCK rank $1$. Under the same hypothesis, not only $J$ is left-invariant, but the entire structure is also left-invariant.

To prove the main theorem we take inspiration from \cite{NagyLengthofHarm}. There, the authors showed that if a Riemannian manifold $M$ has first Betti number $b_1(M) = \dim M - 1$ and all harmonic $1$-forms having constant length must be diffeomorphic to a $2$-step nilmanifold. They considered the Albanese map of $M$ defined over its harmonics $1$-forms to obtain a $S^1$-fibration of $M$ over its Albanese torus $\Alb(M)$. This is possible precisely because its harmonic $1$-forms have constant length. By deforming the metric through some modifications of the induced metric on the Albanese torus, they construct a global frame of $TM$ for which the usual Lie derivative defines a Lie bracket of a $2$-step nilpotent Lie algebra. We can further modify this argument to obtain a complex version of this result for Vaisman manifolds. With this goal in mind we need a version of the Albanese torus suitable for this setting. For a compact complex manifold $M$, its Albanese torus is defined through the holomorphic $1$-forms of $M$. It has a natural structure of a complex torus. When $M$ is K\"ahler it is well known that $\dim_\C \Alb(M)$ coincides with the Hodge number $h^{1,0}(M)$. In general, however, this is not valid. For Vaisman manfolds, Tsukada showed that every holomorphic $1$-forms is basic. We use this fact to prove in section 3 that $\dim_\C \Alb(M) = h_B^{1,0}(M)$ for the basic Hodge number $h_B^{1,0}$. With these tools in hand we obtain the following result, proved in section 5.
\begin{Theo}[Theorem \ref{Constant Length}]
	Let $(M^{2n+2}, J, \omega, \theta)$ be a Vaisman manifold having all basic harmonic $1$-forms of constant length. Assume that $b_1(M) = 2n+1$. Then, $M$ is diffeomorphic to a Kodaira-Thurston manifold where $J$ is the standard left-invariant complex structure under this diffeomorphism.
\end{Theo}
The above theorem cannot be improved for arbitrary $b_1$ by hoping $M$ should be diffeomorphic to a solvmanifold instead. The classical Hopf manifold is a Vaisman manifold having $b_1 = 1$. It has basic harmonic $1$-forms of constant length trivially, yet it cannot be a solvmanifold since it is not aspherical. We give another example for $b_1 > 1$ in Remark \ref{Rem:CounterExample}.

In section 4 we introduce the Vaisman deformations of a Vaisman structure in the sense of Ornea-Slesar \cite{OrnearSlesar}. We show how these deformations are parametrized by a set of basic $1$-forms on $M$. Through this parametrization we show how the complex structures of Vaisman deformations are given by deformation in the large of the original structure. By a result of Rollenske \cite{Rollenske}, for certain nilmanifolds every deformation in the large of a left-invariant complex structure is also left-invariant. In particular, we obtain the following result.
\begin{Theo}[Theorem \ref{LeftInvDef}]
	\mbox{}\
	\begin{enumerate}[label=(\arabic*)]
		\item Every complex structure $J$ given by a Vaisman deformation of $(J_0, g, \theta)$ is a deformation in the large of $J_0$.
		\item Let $M$ be a Kodaira-Thurston manifold with a Vaisman structure $(J, \omega, \theta)$ with left-invariant $J$. Then, the complex structure of every Vaisman deformation of $(J, \omega, \theta)$ is left-invariant. 
	\end{enumerate}
\end{Theo}
Finally, we can use the transverse Calabi-Yau theorem together with a transverse Bochner method to obtain our main theorem with the results discussed above.

In section 7 we discuss a little about transverse geometrically formality (TGF) in the sense of \cite{HabibTGF} in the context of our main results. Geometrically formality was introduced by Kotschick in \cite{Kotschick}. It is a much stronger than the usual de Rham cohomology formality and it is a property of the Riemannian metric. Kotschick showed that geometrically formal metrics shares a lot of properties with flat metrics on compact manifolds. By taking a careful look at \cite[Theorem 10]{sferruzza2025hermitiangeometricallyformalmanifolds} we affirm that, in fact, every flat metric on a compact smooth manifold is geometrically formal. It is natural to consider the transverse version of this property and its relations to the usual geometrically formality, which was stated in \cite{HabibTGF}. Since TGF manifolds have, in particular, basic harmonic $1$-forms of constant length, then it is natural to ask how stronger is TGF in relation to the constant length condition. For Vaisman manifolds with $b_1(M) = 2n +1$ we show that they are equivalent by using the regularity of the foliation.

\mbox{}\\
\textbf{Acknowledgements.} The author is grateful to his supervisor, Hisashi Kasuya, for his helpful comments, suggestions, and illuminating discussions that led to this project.
	
	\section{Preliminaries}

	\subsection{Riemannian Foliations}
	
	Let $(M, \F)$ be a foliated manifold of rank $p$ and codimension $q$. Denote by $T\F$ the distribution associated to the foliation $\F$ and by $Q := TM\slash T\F$ the normal bundle. We have the usual exact sequence of vector bundles,
	\[
	0 \rightarrow L \rightarrow TM \xrightarrow{\Pi} Q \rightarrow 0 
	\]
		
	A vector field $X$ is called {\em foliated} if $[V, X]$ is a section of $T\F$ for all sections $V$ of $T\F$. Let $T$ be a section of $ Q^{\otimes^k} \otimes (Q^*)^{\otimes^l}$. We can consider $\tilde{T}_0 := T \circ \Pi^{\otimes^l}$ a tensor field associated with $T$ on $M$ with values in $Q^{\otimes^k}$. We define:
	
	$$(\Lie_X T) (\Pi X_1, \cdots, \Pi X_l) := (\Lie_X \tilde{T}_0) (X_1,\cdots, X_l).$$

	A straightforward calculation using the usual properties of Lie derivatives of tensor fields and the fact that $T\F$ is involutive shows that $\Lie_X T$ is well defined.

	We say that $T$ is {\em holonomy invariant} if $\Lie_X T = 0$ for all $X\in T\F$. In a foliated local coordinates, the coordinate functions of $T$ and $T_0$ coincides, so holonomy invariance means precisely that the local expression of $T$ on a foliated chart does not depend on the coordinates along the leaves of $\F$.
	A Riemannian bundle metric $g_Q$ on $Q$ is called a {\em transverse Riemannian} metric if $g_Q$ is holonomy invariant.

	\begin{Def}
	 	A Riemannian metric $g$ on $M$ is called {\em bundle-like}, if for any $X, Y \in \mathfrak{X}(M)$ the induced bundle metric $g_Q(\bar{X},\bar{Y}) := g(X^\perp, Y^\perp)$ is a transverse Riemannian metric on $Q$.  We call $(M, \F, g)$ a {\em Riemannian Foliation}.
	\end{Def}
	 
	 Let $g$ be a Riemannian metric on a foliated manifold $(M, \F)$. Let $N := T\F^{\perp}$ be the normal bundle induced by the metric $g$. The bundles $N$ and $Q$ are isomorphic through the usual quotient morphism $\Pi|_N$. Consider the projection morphism $p_N : TM \to N$ and $g_N := g \circ (p_N)^{\otimes^2}$. Notice that $\Pi_N \circ p_N = \Pi$. Suppose $g$ is bundle-like. Then, $g_0 = g_Q \circ \Pi^{\otimes^2} = g_N$ by definition. Hence, $\Lie_X g_N = 0$ for any $X \in \mathfrak{X}(M)$. This remark will be used later.

	 For a foliated manifold $(M, \F)$, an endomorphism $\J : Q \to Q$ is called a transverse almost complex structure if $\J^2 = - \Id_Q$ and $\J$ is holonomy invariant. Furthermore, if for any foliated chart $U$, $\J$ induces an integrable structure on the quotient manifold $\overline{U}$, then $\J$ is called a {\em transverse complex structure}. For a foliated Riemannian manifold, if $\J : Q \to Q$ is a transverse complex structure such that $g_Q(\J \, \bigcdot\,, \J \,\bigcdot ) = g_Q(\bigcdot, \bigcdot )$ and $\omega_0 := g_0((\Pi|_N)^{-1} \, \J \, \Pi \, \bigcdot ,\bigcdot)$ is a closed $2$-form on $M$, the structure $(\F, \omega_0, \J)$ is called a {\em transverse K\"ahler foliation}. 
	 
	 For an oriented Riemannian foliated manifold $(M, \F, g)$ consider the induced orientation on $T\F$ and $Q$. Using the metric bundle $g_Q$, we can consider the transverse Riemannian volume form $\nu_Q$ on $Q$, and its associated $q$-form $\nu :=  \nu_Q \circ \Pi^{\otimes^q}$ on $M$. We also define the characteristic form $\chi$ of $(M, \F, g)$ by	$\chi := \varepsilon^1\wedge \cdots \wedge \varepsilon^p$, where $\varepsilon^i$ is an orthonormal oriented local frame of $T\F^*$. We have that $dV_g = \nu\wedge\chi$.

	Let 
	\[
	\A_B^*(M):= \{\omega \in \A^*(M) \, | \, \iota_X \omega = 0 \text{ and } \Lie_X\omega = 0\}
	\] 
	denote the real valued basic differential forms on $M$. Define $d_B := d|_{\A_B^*(M)}$. The pair $(\A_B^*(M), d_B)$ defines a cochain complex, so we can consider the associated cohomology $H_B^*(M)$. We have a natural homomorphism $H_B^*(M) \to H^*(M)$ by taking a class $[\alpha]_B \mapsto [\alpha] \in H^*(M)$. However, this does not need to be injective nor surjective. We denote by $b_k(\F)$ the dimension of $H_B^k(M)$. 
	
	We define $*_B : \A_B^*( M) \to \A_B^{q - *} (M)$ the basic Hodge star operator by $*_B \eta = (-1)^q *(\eta\wedge \chi)$. The basic Hodge operator satisfies the following equation:  
	$$\omega\wedge *_B\eta = g(\omega, \eta) \nu  \qquad \forall\eta,\omega \in \A_B^*(M).$$
	Assume in addition that $M$ is compact. Consider the following inner product on $\A_B^*(M)$:
	$$(\omega,\eta) := \int_M  \omega \wedge *_B\eta \wedge \chi.$$
	Let $\delta_B$ be the formal adjoint of $d_B$ on $\A_B^*(M)$. We define the basic Laplacian operator $\Delta_B : \A_B^*(M) \to \A_B^*(M)$ to be $\delta_B := d_B \delta_B + \delta_B d_B$. Then, one can consider the space $\mathcal{H}_B^*(M) := \ker \Delta_B$ of harmonic basic forms. In a complete analogy with the Riemannian case, we have the following isomorphism.
	\begin{Theo}[\cite{ElKacimi}, \cite{Tondeur}]
		$H_B^*(M)$ is isomorphic to $\HH_B^*(M)$. 
	\end{Theo}

	We call a foliation $\F$ {\em taut} if there exists a metric which makes all leaves into a minimal submanifold. When $M$ is compact and oriented, tautness can be characterized topologically by saying that $H_B^q(M) \cong \R$. When $\F$ is taut, the basic co-differential can be written as $\delta_B = (-1)^{q(r+1)+1}*_Bd *_B$. For all foliations in this work we assume tautness. See \cite{Tondeur} for more details about these properties.

	Let $(M, \F, \J)$ be a compact foliated manifold with transverse complex structure $\J$. In a similar fashion that is done for complex manifolds, the transverse complex structure induces a splitting $\A_B^k(M)_\C = \bigoplus_{u + v = k}\A_B^{u, v}(M)$. We can define a basic Dolbeault operator $\overline{\partial}_B$ by considering transverse holomorphic charts on $M$, which defines the basic Dolbeault cohomology $H_B^{*, *}(M)$. Take $\overline{\partial}_B^*$ to be the formal adjoint of $\overline{\partial}_B$ with respect to the Hermitian extension of the inner product $(\bigcdot, \bigcdot)$ on $\A_B^*(M)_\C$ and define $\Delta_{\overline{\partial}_B} := \overline{\partial}_B \overline{\partial}_B^* + \overline{\partial}_B^* \overline{\partial}_B$. Then, we define the space of basic Dolbeault forms to be $\HH_B^{*, *}(M) := \ker \Delta_{\overline{\partial}_B}$. Now, suppose $(M, \F, g)$ is a transversely K\"ahler foliated manifold. In a completely similar way to the usual K\"ahler case, by \cite[Theorem 3.3.3]{ElKacimi} and \cite[Theorem 3.4.6]{ElKacimi}, we obtain that $\HH_B^{*,*}(M)$ is a finite dimensional space, there is an isomorphism $H_B^{*, *}(M) \cong \HH_B^{*, *}(M)$ and a decomposition $H_B^{k}(M)_\C = \bigoplus_{u + v = k}H_B^{u, v}(M)$. In particular, basic Dolbeault harmonic forms are basic harmonic.

	\begin{Rem}\label{Rem:RelationOfCandRHarmForms}
		Take a $1$-form $\alpha \in \HH_B^{1,0}(M)$. Then, by the decomposition above $\alpha$ is basic harmonic so we can write $\alpha = \alpha_1 + i\alpha_2$ with $\alpha_1,\alpha_2 \in \HH_B^1(M)$. Now $J\alpha = J\alpha_1 + i J\alpha_2$ and $J\alpha = i\alpha = i\alpha_1 - \alpha_2$, hence $J\alpha_1 = - \alpha_2$. Denote by $A_1$ and $A_2$ the metric dual of $\alpha_1$ and $\alpha_2$, respectively. Therefore, we obtain that $g(J A_1, \bigcdot) = - g(A_1, J \bigcdot) = - \alpha_1 \circ J = J \alpha_1 = -\alpha_2 = -g(A_2, \bigcdot)$, hence $JA_1 = - A_2$. 
	\end{Rem}

	We end this subsection with the following definition.
	\begin{Def}
		For a compact foliated manifold $(M, \F)$, we say that $\F$ is {\em quasi-regular} if all its leaves are compact. We say that $\F$ is {\em regular} if $M \slash \F$ has a natural structure of a smooth manifold.   
	\end{Def}
	When a foliation is regular, it is straightforward to verify that all basic forms and any transverse structure descends naturally to the quotient manifold $M \slash \F$.

	\subsection{Vaisman Manifolds}\label{SubSec:VaismanandSasakian}
	
	Let $(M, J, g)$ be a Hermitian manifold with complex structure $J$ and fundamental form $\omega := g(J \bigcdot, \bigcdot)$ with complex dimension $n \geq 2$. We often denote by $M^{2n}$ to indicate that $M$ has real dimension $2n$. From hereon, we consider all manifolds to be connected and compact. If there exists a closed 1-form $\theta$ such that $d\omega = \theta \wedge \omega$, then the triple $(J, \omega, \theta)$ is called a {\em Locally Conformally K\"ahler} (LCK) structure on $M$. We call this structure {\em Vaisman} if $\theta$ is parallel with respect to the Levi-Civita connection, $\nabla^g \theta = 0$. We implicitly assume that the Lee form is not exact. In this way we exclude the K\"ahler case, unless explicitly stated. In the same way, we always assume that $|\theta|_g = 1$. 
	
	Denote by $U := \theta^{\#}$ and $V := JU$. We also denote by $\theta^c := J\theta$, so that $(\theta^c)^{\#} = V$. 
	
	\begin{Theo}[\cite{VaismanGHopfMfd}]
		Let $(M, J, \omega, \theta)$ be a Vaisman manifold. Then, $$\omega = - d\theta^c + \theta\wedge \theta^c.$$
	\end{Theo}
	\begin{Rem}\label{SignConvRem}
		In \cite[Theorem 3.1]{VaismanGHopfMfd} Vaisman derives that $\omega =  d\theta^c - \theta\wedge \theta^c$. This is due to his convention that the fundamental form of a Hermitian manifold $(M, J, g)$ is given by $\omega = g(\bigcdot, J\bigcdot)$, which is the opposite of ours. The authors in \cite{OrnearSlesar} uses the same convention as Vaisman.
	\end{Rem}

	We collect some of the properties of $U$ and $V$.
	
	\begin{Prop}[\cite{VaismanGHopfMfd}, \cite{TsukadaHolomorphicForms}]
		\mbox{}
		\begin{itemize}
			\item $[U, V] = 0$.
			\item $U$ and $V$ are real holomorphic vector fields.
			\item $U$ and $V$ are Killing vector fields.
			\item The induced Riemannian foliation $\Sigma$ generated by $U$ and $V$ is transversely K\"ahler.
			\item The transversely K\"ahler metric is given by $\omega_0 = - dJ\theta$.
		\end{itemize}
	\end{Prop}
	
	The foliation $\Sigma$ above is called the {\em characteristic foliation} of the Vaisman manifold $(M, J, \omega, \theta)$.
	
	\begin{Prop}[\cite{TsukadaCanonicalFoliation}]\label{Prop:CompactLeaf}
		The characteristic foliation $\Sigma$ of any compact Vaisman manifold has a compact leaf.
	\end{Prop} 
	
	Now, from \cite[Theorem 10.22]{OrneaVerbitsky} and the above properties of $U$ and $V$ we obtain the following result.
	
	\begin{Prop}\label{Prop:QuasiRefVaismanM}
		Let $(M, J, \omega, \theta)$ be a Vaisman manifold. If  $\Sigma$ is quasi-regular, then $\Sigma$ is defined through the holomorphic action of a complex torus $T^2$ defined by the flows of $U$ and $V$.
	\end{Prop}
	
	\begin{Theo}[\cite{ChenPiccini}, \cite{Vaisman0}]
		Let $(M, J, \omega, \theta)$ be a Vaisman manifold. If $\Sigma$ is regular, then $M$ is a principal $T^2$-bundle over a compact K\"ahler manifold $\X$
		$$ T^2 \longrightarrow M \xlongrightarrow{} \X.$$
	\end{Theo}
	
	For a lack of better terminology, we call the above fiber bundle presentation of $M$ the {\em Boothby-Wang fibration} of $M$. We often abuse notation and denote by $\omega_0$ both the transverse K\"ahler form and the induced form on the quotient space $\X$ when $\Sigma$ is regular.

	In Vaisman geometry, we have a complete description of basic harmonic forms in terms of harmonic forms in the manifold. In this work, we only need the characterization for $1$-forms stated in the proposition below. 
	
	\begin{Prop}[\cite{KashiwadaHarmonicForms}, \cite{VaismanGHopfMfd}]\label{Prop:Kashiwada}
		Let $(M, J, \omega, \theta)$ be a compact Vaisman manifold. Then, $\HH^1(M)_\C = \HH_B^1(M)_\C \oplus \C \theta$. In particular, $H^1(M)_\C = H_B^*(M)_\C\oplus \C [\theta]$. 
	\end{Prop}
	
	In the same way, Tsukada showed that holomorphic $p$-forms are basic for $p < \dim_\C M$. Again, we only state the result for $1$-forms.
	
	\begin{Prop}[\cite{TsukadaHolomorphicForms}]\label{Prop:Tsukada}
		Let $(M, J, \omega, \theta)$ be a compact Vaisman manifold. Then, every holomorphic $1$-form is a closed basic form.
	\end{Prop}
	
		Since $\theta$ is a closed $1$-form, it defines a cohomology class $[\theta] \in H^1(M, \R)$. Now, if $M$ is compact we can take a basis of integral classes $\alpha_1, \dots, \alpha_l$ of $H^1(M, \R)$ and write $[\theta] = \sum_j a_j \alpha_j$. Now, consider the de Rham-Weil isomorphism $\varphi : H^1(M, \R) \to H_1(M)^*$ given by $\varphi (\alpha) := \int_{\bigcdot} \alpha$. The following definition is going to be essential for section \ref{TheMappingTorus}.
	\begin{Def}
		The LCK rank $r$ of an LCK structure is defined to be
		\[
		 r := \dim_\mathbb{Q} (\Span_\mathbb{Q} \varphi([\theta]) (H_1(M))).
		 \]
	\end{Def}

	A Sasakian manifold is a Riemmanian manifold $(S, g)$ of odd dimension such that the Riemannian cone $(\R^{>0} \times S, d t^2+ t^2g)$ has a complex structure such that the metric is K\"ahler. In fact, Sasakian manifolds provide a source of examples of Vaisman manifolds by considering the $\Z$ action of a holomorphic homothety $h_\lambda (t, p) := ( \lambda t, p)$, for $\lambda > 0$, on $\R^{>0} \times M$. The quotient $\Z \backslash (\R^{>0} \times M)$ is naturally endowed with a Vaisman structure. In particular, notice that in this case $M$ is diffeomorphic to $S^1\times S$. We call such examples the trivial Vaisman extension of $S$. Another important set of examples of Vaisman manifolds are the Kodaira-Thurston nilmanifolds defined below.

	 \begin{Def}\label{HeinsbrgDef}
	 	We define the Heisenberg group $H_{2n+1}$ of dimension $2n+1$  to be the following Lie group
	 	$$H_{2n+1} := \left\{
	 	\begin{pmatrix}
	 		1&  x_1& x_2 & \cdots & x_n  & z \\
	 		&  1&  &  &  & y_1 \\
	 		&  &  \ddots&   &  & y_2 \\
	 		&  &  & \ddots &  & \vdots  \\
	 		&  &  &  & 1 & y_n \\
	 		&  &  &  &  & 1 \\
	 	\end{pmatrix} \ \middle| \ x_i,y_i,z \in \R \right\}$$
	 	It is a 2-step nilpotent Lie group admitting a lattice. It also has a natural left-invariant Sasakian structure.
	 \end{Def}
	 
	 A {\em solvmanifold} $M = \Gamma \backslash G$ is a compact quotient of a connected, simply-connected solvable Lie group $G$ by a lattice $\Gamma$. When $G$ is nilpotent we call $M$ a {\em nilmanifold}. For $G = \R \times H_{2n+1}$ we call $M$ a {\em Kodaira-Thurston} manifold. 
	 	 
	 Consider the Kodaira-Thurston manifold $M = \Gamma\backslash(\R \times H_{2n+1})$. Let $\g$ be the Lie algebra of $\R\times H_{2n+1}$. We have $\g =  \R T\oplus \,\Span_\R \,\{X_i,Y_i,Z\}_i$ where $X_i, Y_i, T, Z$ are vectors such that $[X_i, Y_i] = Z$ with $T$ and $Z$ central. We can define a left-invariant complex structure $J$ on $M$ by setting $JX_i := Y_i$ and $JZ = -T$. Now, consider the left-invariant forms on $M$ defined by
	 \[ \omega := \sum_i X_i^*\wedge Y_i^* - T^*\wedge Z^* \qquad \theta := T^*. \]
	 The triple $(J, \omega, \theta)$ defines a Vaisman structure which we call the {\em standard} Vaisman structure of $M$.

	\subsection{Basic Vector Bundles and the Basic Chern Classes}
	Let $(M, \F)$ be a foliated manifold and set $\mathbb{F} = \C$ or $\R$. Let $\pi : E \to M$ be a $\mathbb{F}$-vector bundle of rank $r$. We call $E$ a {\em basic vector bundle} if there exists an open covering $\{U_\alpha\}_\alpha$ and a family of local trivializations $\{ \psi_\alpha : \pi^{-1} (U_\alpha) \to U_\alpha \times \mathbb{F}^r\}_\alpha$ such that its induced family of transition functions $f_{\alpha\beta} : U_{\alpha}\cap U_\beta \to GL(\mathbb{F}^r)$ are basic functions, i.e., $\iota_X df_{\alpha\beta} = 0$ for all $X \in T\F$ and for all $\alpha,\beta$. We call $\psi_\alpha$ a basic local trivialization.

	A $\mathbb{R}$-bilinear operator $D : \Gamma T\F \times \Gamma E \to \Gamma E$ is called a {\em partial flat connection} on $E$ if for any $f \in \mathcal{C}^\infty(M, \mathbb{F})$, $g \in \mathcal{C}^\infty(M, \R)$, $s \in \Gamma E$ and $X, Y \in \Gamma T\F$ 
		\begin{enumerate}[label=(\arabic*)]
			\item $D_X (fs) = X(f)s + fDs$;
			\item $D_{gX} s = gD_X s$;
			\item $R^D (X,Y) := [D_X, D_Y] - D_{[X, Y]} = 0$.
		\end{enumerate}
	Given a basic $\mathbb{F}$-bundle $\pi : E \to M$, we can define a natural partial flat connection $D$ by declaring all local frame induced by the basic local trivializations of the definition above to be parallel sections of $E$. In fact, the existence of such partial connection defines the basic structure of $E$.
	
	\begin{Theo}[\cite{Rawnsley}]
		A $\mathbb{F}$-vector bundle $E$ is a basic if and only if there exists a partial flat connection $D$ on $E$.
	\end{Theo}
	For a basic $\mathbb{F}$-vector bundle $E \to M$ and its partial flat connection $D$, we say that a local section $s$ of $E$ is {\em basic} if $Ds = 0$. Consider another basic $\mathbb{F}$-vector bundle $F \to M$. In the same way, a section $T$ of $(E^*)^{\otimes^k}\otimes F$ is called {\em basic} if, for any basic local sections $s_1, \dots, s_k$ of $E$,  $T(s_1, \dots, s_k)$ is a basic section of $F$.
	
	Let $E$ be a basic $\C$-vector bundle of rank $r$. Given a connection $\nabla$ on $E$, for a choice of a local frame $S := (s_1, \dots, s_r)$, take $\theta := (\theta_j^i)_{ij}$ to be the connection matrix of $1$-forms of $\nabla$, i. e., $\nabla s_j = \sum_i \theta_j^i \otimes s_i$. We say that $\nabla$ is a {\em basic connection} on $E$ if for all basic local frame $S$, $\theta_j^i$ is a basic form for all $i, j$.
	
	In the same way, we can consider the curvature $2$-form $F_\nabla$ of a connection $\nabla$ given by
	$$F_\nabla = d\theta + \frac{1}{2}\theta\wedge \theta.$$
	This is a globally defined closed $2$-form with values in $\End E$. Now, suppose $\nabla$ is a basic connection. Then, $F_\nabla$ is a basic $2$-form with values in $\End E$. In particular, the curvature matrix of $2$-forms associated to a basic local frame $S$ is also basic. This implies that the usual Chern forms $c_i(\nabla) \in \A^{2i}(M)_\C$ are closed basic forms. 
	
	\begin{Lemma}
		Consider two basic connections $\nabla$ and $\tilde{\nabla}$ on $E$. Then, $[c_i(\nabla)]_B = [c_i(\tilde{\nabla})]_B \in H_B^{2i}(M)_\C$.
	\end{Lemma}
	\begin{proof}
		We have that $A := \nabla - \nabla' \in \A^1(M, \End E)$. For a basic local frame $S$ of $E$, $A$ is given locally by $(\theta_j^i - \tilde{\theta}_i^j)_{ij}$ which is a matrix of basic $1$-forms. Thus, $A$ is a basic $1$-form with values in $\End E$. With this observation, the result follows by the same argument in the proof of \cite[Lemma 4.4.6]{Huybrechts}.
	\end{proof}

	\begin{Def}
		Let $E \to M$ be a $\C$-vector bundle and suppose $E$ admits a basic connection $\nabla$. We define the basic Chern classes by $c_{i, B}(E) := [c_i(\nabla)]_B \in H_B^{2i}(M)_\C$.
	\end{Def}
	
	Consider a Riemannian foliated manifold $(M, \F, g)$, with $g$ bundle-like. Let $Y \in \Gamma N$. We can define the {\em transverse connection} $\nabla^T$ on $N$ in the following way:
	\begin{equation*}
		\nabla^T_X Y := \begin{cases*}
			(\nabla^g_X Y)^\perp & if  $X \in \Gamma N$\\
			[X, Y]^\perp & if  $X \in \Gamma T\F$
		\end{cases*}
	\end{equation*}
	where $\nabla^g$ is the Levi-Civita connection of $(M, g)$. Under the isomorphism $N \cong Q$ induced by $g$, $\nabla^T$ is the {\em transverse Levi-Civita connection} as defined in \cite[Chapter 5]{Tondeur}.

	\begin{Prop}
		The normal bundle $N$ is a basic $\R$-vector bundle and $\nabla^T$ defined above is a basic connection on $N$.
	\end{Prop}
	\begin{proof}
		For sections $Y \in \Gamma N$ and $X \in \Gamma T\F$, define $D_X Y := [X, Y]^\perp$. This is the partial Bott connection of $N$ which is a flat partial connection \cite[Chapter 3]{Tondeur}. Notice that $D Y = 0$ if and only if $Y$ is a foliated vector field. From \cite[Theorem 5.11]{Tondeur}, we know that $\nabla^T$ satisfies the following condition:
		$$(\Lie_Z \nabla^T_X Y)^\perp = \nabla^T_{[Z, X]} Y + \nabla^T_X (\Lie_Z Y)^\perp \qquad \forall X \in \Gamma TM, Y \in \Gamma N, Z \in \Gamma T\F .$$
		Suppose $Y$ is a foliated vector field on some open set $U \subset M$. Assume further that $Y$ is a local section of $N$. Thus, $(\Lie_Z Y)^\perp = 0$ yields $(\Lie_Z \nabla^T_X Y)^\perp = \nabla^T_{[Z, X]} Y$. Let $Y_1, \dots, Y_k$ be a local frame of $N$ composed of foliated vector fields. Then, $\nabla^T Y_j =  \sum_i\theta^i_j \otimes Y_i$. Hence, $\Lie_Z \nabla^T Y_j =\sum_i ( (\Lie_Z \theta^i_j) \otimes Y_i + \theta^i_j \otimes [Z, Y_i])$, which implies that $(\Lie_Z \nabla^T Y_j)^\perp = \sum_i (\Lie_Z \theta^i_j) \otimes Y_i$. If $X$ is a local section of $T\F$, then $[Z, X] \in \Gamma  T\F$, hence $\nabla^T_{[Z, X]} Y_i = 0$. If $X \in \Gamma N$ is a basic section, then $[Z, X] \in \Gamma T\F$, hence $\nabla^T_{[Z, X]} Y_i = 0$. These two facts together yield $\Lie_Z \theta^i_j = 0$. On the other hand $\nabla_Z^T Y_i = [Z, Y_i]^\perp = 0$, since $Y_i$ is foliated. Thus, $\theta^i_j (Z) = 0$ for any local section of $T\F$.

	\end{proof}
	
	Suppose $(M, \F, g)$ is a foliated Riemannian foliation and assume that $(\F, \omega_0, \J)$ is a transverse K\"ahler foliation on $M$. Identify $N$ with $Q$ through the isomorphism induced by $g$. By taking the complexification $N \otimes \C$ of the normal bundle $N$ and considering the $\C$-extension of $\nabla^T$ and $\J$, we obtain the eigenspace decomposition $N\otimes \C = N^{1,0} \oplus N^{0,1}$ with respect to $\J$. By a similar argument as in the K\"ahler case, one can show that $\nabla^T|_{N^{1,0}}$ restricts to a connection on $N^{1,0}$, which is also basic as a consequence of the previous proposition.
	
	\begin{Def}
		Let $(M, \F, g, \J)$ be a foliated Riemannian manifold with a transverse K\"ahler structure $(\F, \omega_0, \J)$. We define by $c_{i, B}(\F) := c_{i, B}(N^{1,0})$ the basic Chern classes of the foliation $\F$.
	\end{Def}
	
	Denote by $R^{\nabla^T} \in \A^2(M, \End N)$ the curvature $2$-form of $\nabla^T$. We can extend it to have values in $\End TM$ by defining $R^{\nabla^T}(X, Y) Z := 0$ whenever $Z$ is a vector in $T\F$. For any point $p\in M$ and $X, Y \in T_pM$ define $\Ric^T(X, Y) := \Tr (Z \mapsto  R^{\nabla^T}(Z, X)Y)$. If $Y \in T_p\F$, then $\Ric^T(X, Y) = 0$ by construction. If $X \in N_p$, then $\Ric^T(X, Y) = 0$ by \cite[Corollary 5.12]{Tondeur}. Once again, by \cite[Corollary 5.12]{Tondeur} we can conclude that $\Ric^T$ is a basic symmetric $2$-tensor. Similarly to the K\"ahler case, if $(\F, \omega_0, \J)$ is a transverse K\"ahler foliation, $\Ric^T$ is $\J$-invariant, therefore $\rho^T := \Ric^T(\J \bigcdot, \bigcdot)$ is the transverse Ricci form associated with $\omega_0$. 
	
	\begin{Prop}
		Let $(M, \F, \omega_0, \J)$ be a transversely K\"ahler foliated manifold. Then, $\rho^T = 2\pi c_1(\nabla^T|_{N^{1,0}})$. In particular, $[\rho^T]_B = 2 \pi c_{1,B}(\F)$.
	\end{Prop}
	As before, the proof is analogous to the K\"ahler case. For more on basic Chern classes and applications see \cite{biswas2023sasakiangeometryheisenberggroups, BiswasKasuya1, BiswasKasuya2}.

	\section{The Albanese Torus and the Albanese Map of Vaisman Manifolds}\label{Sec:AlbTorus}
	
	For every compact complex manifold $M$ one can define an associated complex torus in the following way. Denote by $\Omega^p(M)$ the space of holomorphic $p$-forms of $M$ and define a group homomorphism $\varphi : H_1(M) \to (\Omega^1(M))^*$ by
	$$\varphi([\gamma])=\left.\int_\gamma  \,\right|_{\Omega^1(M)}.$$
	Recall that for a compact complex manifold $\Omega^p(M) = H^{p,0}(M)$ are finite dimensional complex vector spaces.
	
	\begin{Def}
		Let $M$ be a compact connected complex manifold. Consider $\Delta := \overline{\varphi(H_1(M))}$ the closure of $\varphi(H_1(M))$ in the finite dimensional vector space $(\Omega^1)^*$. The space $\Alb(M) := \faktor{(\Omega^1(M))^*}{\Delta}$ is called the {\em Albanese torus} of $M$. Fix a point $x_0 \in M$. The map $\alpha_{x_0} : M \to \Alb(M)$ defined by
		$$\alpha_{x_0}(x) := \left[\int_{x_0}^{x} \right]$$
		is well-defined and called the {\em Albanese map} of $M$.
	\end{Def}
	When it is clear from the context, we omit the fixed point $x_0$ from our notation. As the name suggests $\Alb (M)$ so defined is a torus.
	
	\begin{Prop}[Theorem 9.7 in \cite{Ueno}]
		\mbox{}\
		\begin{itemize}
			\item The Albanese Torus of $M$ is a complex torus.
			
			\item The Albanese map $\alpha_{x_0}$ is a holomorphic map.
			
			\item The pair $(\Alb M, \alpha_{x_0})$ satisfies a universal property as follows.

		\end{itemize}
	Let $F : M \to T$ be a holomorphic map from $M$ to a complex torus $T$. Then, there exists a unique holomorphic map $\tilde{F} : \Alb(M) \to T$ such that the diagram below commutes
	\[
	\begin{tikzcd}
		M \arrow[r, "F"] \arrow[d, "\alpha_{x_0}"'] & T \\
		\Alb (M) \arrow[ru, "\tilde{F}"', dashed]         &        
	\end{tikzcd}
	\]
	\end{Prop}

	Recall that the rank of a finitely generated Abelian group $G$ is given by the number of generators, from any minimal generating set, which are not torsion elements of $G$.
	\begin{Theo}
		Let $M$ be a Vaisman manifold. Then $\Alb(M)$ is a complex torus with $\dim_\C \Alb(M) = \dim_\C H_B^{1,0}$.
	\end{Theo}
	\begin{proof}
		Write $k := \dim_\C H_B^{1,0} = \dim_\C \Omega^1(M)$. Let $\{\omega_1, \dots, \omega_k\}$ be a basis of the $\C$-vector space $\Omega^1(M)$. By Proposition \ref{Prop:Kashiwada} $\rk H_1(M) = 2k +1$. Our goal is to show that $\varphi(H_1(M))$ is a lattice on $(\Omega^1(M))^*$ as a $\R$-vector space. 
		
		\textbf{Step 1}: Showing that $\rk \ker \varphi = 1$.
		
		Notice that $\varphi$ annihilates torsion elements. We need to find a non-torsion, nontrivial element of $H_1(M)$ which generates $\ker \varphi$ together with the torsion elements. Define $H_\C := H_1(M)\otimes_\Z \C$  and $H_\R := H_1(M)\otimes_\Z \R$. Observe that $H_\R$ can be considered as a subset of $H_\C$ composed by all elements $\sigma \in H_\C$ such that $\sigma = \overline{\sigma}$. Tensoring with $\R$ or $\C$ kills off torsion elements. Thus, it is enough to show that $\dim_\R \ker \varphi \otimes_Z \R = 1$. 
		
		\textbf{Step 1.1}: Finding a non-torsion nontrivial $\gamma \in H_1(M)$ such that $\gamma \in \ker \varphi$.
		
		Denote by $\Sigma$ the characteristic foliation of $M$ and $T\Sigma$ the tangent distribution. Since $M$ is compact Vaisman, there exists a compact leaf $L\subset M$ by Proposition \ref{Prop:CompactLeaf}. Consider $G = \C$ the action on $M$ given by the flows of $U$ and $V$. Fix a $p \in L$ and let $\phi_p : G \to L$ be the map $\phi_p(g) := g\bigcdot p$. Since $G$ and $L$ have the same dimension, $G$ acts transitively on $L$ and $\phi_p$ is surjective, thus $\phi_p$ is a smooth covering map. Hence, $L$ is biholomorphic to a complex torus and $G_p$ is a lattice of $G$.

		Let $\{U_0, V_0 \}\subset G_p$ be a minimal generating set of $G_p$. Then, the lines $tU_0$ and $tV_0$ are projected into circles on $L$. Now, either $\theta(U_0) \neq 0$ or $\theta(V_0) \neq 0$, otherwise $\theta$ would be zero on $T\F$ which is absurd. Suppose $\theta(U_0) \neq 0$. Since $G$ acts by isometries, $\theta(U_0) = g(U, U_0)$ which is a constant, because the flow of $U$ is given by geodesics and $U_0$ is a Killing vector field. Let $\gamma(t) := tU_0\bigcdot p$. As showed above, this is a parametrization of a circle in $M$. By abuse of notation, we denote by the same $\gamma$ the class induced in $H_1(M)$, $H_\C$ or $H_\R$. 
		
		Consider $\psi : H^1(M)_\C \to \Hom (H_1(M), \C)$ the de Rham isomorphism given by integration $\psi(\eta) := \int_{\bigcdot} \eta$. For any element $f\in \Hom (H_1(M),  \C)$ we can map it to $f_\C \in \Hom (H_\C, \C) = H_\C^*$ by defining $f_\C (\sigma \otimes a) := a f_\C(\sigma)$ for all $a \in \C$. This defines a isomorphism of $\C$-vector spaces, since any morphism $f : H_1(M) \to \C$ annihilates torsion elements. Thus, we can consider the isomorphism $\psi : H^1(M)_\C \to H_\C^*$. Hence, we have that the transpose $\psi^t : H_\C^{**} \to H^1(M)_\C^*$ is also an isomorphism. The vector space $H_\C$ is a finite dimensional, so we can identify it with $H_\C^{**}$ and assume that $\psi^t : H_\C \to H^1(M)_\C^*$. A straightforward calculation gives $b_\gamma := \psi(\theta)[\gamma] = \int_\gamma \theta \neq 0$ and $\int_\gamma \omega_i = 0$ for all $i$ since $\omega_i$ are all basic. Thus, $\gamma$ defines a nontrivial class in the homology group $H_1(M)$. Notice that, $\int_{k\gamma} = k\int_{\gamma} \neq 0$ for all $k\in\Z\setminus 0$, hence $\gamma$ does not define a torsion element in $H_1(M)$. 
		

		\textbf{Step 1.2}: Showing that $(\ker \varphi) \otimes_\Z \R = \R \gamma$.

		Indeed, let $\sigma \in H_1(M)$ be a nontrivial, non-torsion element such that $\varphi(\sigma) = 0$. Then, for any $\eta \in \Omega^1(M) = H_B^{1,0}(M)$ we get that $0 = \varphi(\sigma)[\eta] = \int_\sigma \eta = \psi^t(\sigma\otimes 1)[\eta]$. Taking the conjugate of this expression gives $0 = \int_\sigma \overline{\eta}$, which implies that $\psi^t(\sigma\otimes 1)[\eta] = 0$ for all $\eta \in H_B^{0,1}(M) = \overline{H_B^{1,0}(M)}$. Finally,
		$a := \psi^t(\sigma)[\theta] = \int_\sigma \theta\neq 0$ otherwise $\sigma$ would be either trivial or a torsion element. Thus, $\sigma = a/b_\gamma \gamma$ which is in $H_\C$. Since $a$ and $b_\gamma$ are real numbers, we conclude that $\sigma \in H_\R$.

		\begin{Claim}
			$\rk \ker \varphi = 1$.
		\end{Claim}
		Suppose $\rk \ker \varphi > 1$. Then, we can take at least two non-torsion elements $\sigma_1$ and $\sigma_2$ from a minimal generating set. Since tensoring with a field maps a minimal generating set to a basis, $\sigma_1$ and $\sigma_2$ are linearly independent. Absurd, since $\ker \varphi \otimes_\Z \R$ is one dimensional.
		
		\textbf{Step 2}: Showing that $\dim_\R( \Span_\R \varphi(H_1(M))) = 2k$.
		
		From the previous step we can write $\ker \varphi = \langle\sigma \rangle\oplus \text{Torsion}$, where $\sigma$ in a non-torsion nontrivial element of $H_1(M)$. Now, consider $\tilde{\varphi} := H \to (\Omega^1(M))^*$ the induced group isomorphism on the quotient space $H := H_1(M)\slash (\ker \varphi)$. Since $H_1(M)$ is a finitely generated Abelian group, then so is $H$. Notice that $H$ is torsion free, thus we can consider a minimal generating set $\{\overline{\sigma_1}, \dots, \overline{\sigma_{2k}}\}$ for $H$, where $\sigma_j \in H_1(M)$ for all $j$. Finally, let $a_1, \dots, a_{2k} \in \R$ such that $f := a_1\tilde{\varphi}(\overline{\sigma_1}) + \cdots + a_{2k}\tilde{\varphi}(\overline{\sigma_{2k}}) = 0$. Hence, for all $i$
		\begin{align*}
			0 = f(\omega_i) & = a_1\int_{\sigma_1} \omega_i 	+ \cdots + a_{2k} \int_{\sigma_{2k}} \omega_i\\
			& = \psi^t(a_1 \sigma_1 + \cdots + a_{2k} \sigma_{2k})[\omega_i]
		\end{align*}
		By taking the conjugate of the above expression we obtain that $\psi^t(a_1 \sigma_1 + \cdots + a_{2k} \sigma_{2k})[\overline{\omega_i}] = 0$ for all $i$. If $ a_1 \sigma_1 + \cdots + a_{2k} \sigma_{2k}$ is not trivial in $H_\R$, then $c := \psi^t(a_1 \sigma_1 + \cdots + a_{2k} \sigma_{2k})[\theta] \neq 0$, thus $a_1 \sigma_1 + \cdots + a_{2k} \sigma_{2k} = c/b_\gamma \gamma$ in $H_\R$. This is impossible since $\{\sigma_1,\dots,\sigma_{2k}, \gamma\}$ form a basis for $H_\R$ by construction. Thus, $a_1 \sigma_1 + \cdots + a_{2k} \sigma_{2k} = 0$ in $H_\R$. By passing to $H \otimes_\Z \R$ we obtain that $a_j = 0$ for all $j$.

		
	%
	\end{proof}
	
	In our context, since $\varphi(H_1(M))$ defines a lattice on $(\Omega^1(M))^*$ this allows us to compute the differential of the Albanese map directly.
	\begin{Prop}
		Let $M$ be a Vaisman manifold. Then, $d(\alpha_{x_0})_{x}[v] = \iota_v|_{\Omega^1}$.
	\end{Prop}
	\begin{proof}
		Fix a point $p \in M$ and let $v\in T_pM$. Take $\gamma$ to be a small curve in $M$ such that $\gamma(0) = p$ and $\dot{\gamma}(0) = v$. Fix a smooth path $\sigma$ connecting $x_0$ and $p$. Then,		
		$$\alpha_{x_0} \circ \gamma(t) = \left[ \int_\sigma  \right] + \left[ \int_{\gamma(0)}^{\gamma(t)} \right].$$
		Consider $\omega_1, \dots, \omega_k$ a basis of $\Omega^1(M)$ and denote by $F_1, \dots, F_k$ the dual basis. By taking a small open simply-connected neighborhood $U$ of $p$, we can suppose that for all $i$, $\omega_i|_U = df_i$ for some holomorphic function $f_i : U \to \C$, because $\omega_i$ are closed by Proposition \ref{Prop:Tsukada}. Then,
		$$ \int_p^{\gamma(t)}  = \sum_i \left(\int_p^{\gamma(t)} \omega_i \right) F_i = \sum_i ((f_i\circ \gamma(t) - f_i (p)) F_i) \in (\Omega^1(M))^*.$$
		This implies that
		$$ \left.\frac{d}{dt}\right|_{t = 0}  \int_p^{\gamma(t)}  = \sum_i \omega_i(v) F_i.$$
		By identifying locally at $\alpha_{x_0}(p)$, $\Alb (M)$ with $(\Omega^1(M))^*$, we obtain that  $d(\alpha_{x_0})_{x}[v] = \iota_v|_{\Omega^1}$.
	\end{proof}
	
	\section{Deformation of Vaisman Structures}\label{DefOfVaismSection}
	
	Let $M$ be a connected and compact smooth manifold. Consider the set of all complex structures on $M$	
	$$\mathcal{C} := \{ J \in \Gamma\End (TM) \, | \, J \text{ is a complex structure}\}.$$
	
	\begin{Def}
		Endow $\mathcal{C}$ with the subspace topology from $\Gamma\End (TM)$. We say that $J'$ is a {\em deformation in the large} of a complex structure $J \in \mathcal{C}$ if $J'$ is in the same connected component of $J$.
	\end{Def}
	 
	 Let $(M, J, g)$ be a Hermitian manifold and $(\F, g)$ a Riemannian foliation with $g$ bundle-like. If $J$ preserves $T\F$, then $J$ also preserves $N := T\F^\perp$. The complex structure induces a transverse complex structure $\J$ on $Q$ by taking $\J \, \bar{X} := \overline{JX}$. Indeed, first observe that this is well-defined, since taking $\overline{X} = \overline{Y}$ implies that $X - Y \in T\F \implies J(X-Y) \in T\F \implies \overline{J(X-Y)} = 0$. Through the identification $Q \simeq N$, we can consider $J_N:= \J = J|_N$. In this case we can consider the Nijenhuis tensor, $N_{J_N} := (N_J|_N)^\perp$ of $J_N$. Given any foliated chart $U$, $J_N$ descends to an almost complex structure on $\overline{U}$. In this case, $N_{J_N}$ descends to the usual Nijenhuis tensor of an almost complex structure. However, $N_J = 0$ so that $N_{J_N} = 0$, which implies that $J_N$ is indeed a transverse complex structure.
	 
	 Let $(M, J, g, \theta)$ be a Vaisman manifold. We define {\em Vaisman deformations} in the sense of Ornea-Slesar \cite{OrnearSlesar} of these structures in the following. Define the set 
	 $$\mathcal{V}(J, g, \theta) := \{\text{Vaisman structures }(J', g', \theta') \,|\, \theta' = \theta, U' = U, V' = V \text{ and } \J' = \J\}.$$
	 This is the Vaisman version of type II deformations of Sasakian structures as described in \cite[Section 7.5.1]{BoyerGalicki}. In \cite{OrnearSlesar} the authors shows that $\mathcal{V}$ above is always nontrivial as long as a non-zero global basic smooth function $\varphi$ exists on $M$.

	Consider a Vaisman structure $(J', g', \theta') \in \mathcal{V}(J, g, \theta)$. Write $\zeta := \theta'^c - \theta^c \in \A_B^1(M)$. We want to find an expression of $J'$ in terms of $J$. Notice that $\bar{J'}\bar{X} - \bar{J}\bar{X} = 0 \implies \overline{J'X - JX} = 0 \implies J'X - JX = aU+bV$ for some $a, b \in \R$. By a straightforward calculation we obtain that  
	 \begin{itemize}
	 	\item $\theta(J'X - JX) = - \zeta(X)$,
	 	\item $\theta^c(J'X - JX) = - \zeta(J X)$.
	 \end{itemize}
	 Now, $\theta(X) = g'(X, U)$ and $\theta'^c(X) = g'(X, V)$ yields $\theta(J'X - JX) = a$ and $\theta'^c(J'X - JX) = b$. Thus,
	 \begin{equation}\label{eqn:J}
	 	 J' = J  - \zeta \otimes U  - (\zeta \circ J )\otimes V. \tag{$*$}
	 \end{equation}
	For any basic $(1,1)$-form $\alpha$, we say that $\alpha$ is positive-definite with respect to $J$ if $\alpha(X, JY) > 0$ for all $X, Y \in N$. We denote it by $\alpha > 0$ w.r.t. $J$ (on $N$).

	 \begin{Prop}\label{VaismanDefProp}
	 	Consider the set 
	 	$$C(M, J, \omega_0) := \{ \zeta \in \mathcal{A}^1_B(M, \R) \,|\, \omega_0 - d\zeta > 0 \text{ w.r.t }J \}.$$
	 	Then, we have a bijection $C(M, J, \omega_0) \to \mathcal{V}(J, g, \theta)$.
	 \end{Prop}
	 \begin{proof}
	 	We follow the work of Ornea and Slesar in \cite{OrnearSlesar} (with a different sign convention. See Remark \ref{SignConvRem}). Let $\zeta$ be a basic $1$-form. Define $J'$ by the expression \ref{eqn:J}. Consider $\theta'^c := J'\theta = \theta^c + \zeta$ and define $g' := -d\theta'^c(\bigcdot, J\bigcdot) - d\zeta(\bigcdot, J\bigcdot)$. The authors then impose along the way some sufficient conditions on $\zeta$ to guarantee that the deformed structure $(J', g', \theta)$ is in $\mathcal{V}(J, g, \theta)$. A careful reading shows that the main issues are to ensure that $J'$ is a complex structure and $g'$ a Riemannian metric. If this holds, expressions (3.8) and (3.10) in \cite{OrnearSlesar} implies that $g'$ is a Vaisman metric. 
	 	
	 	Now, assume that $\zeta \in C(M, J, \omega_0)$. By \cite[Proposition 3.2]{OrnearSlesar}, for $J'$ defined by \ref{eqn:J} be a complex structure, it suffices for $d\zeta$ to be a $(1,1)$-form with respect to $J$. Since $\omega_0' := \omega_0 - d\zeta > 0$ with respect to $J$, we have that for any non-zero $v \in N$, $\omega_0'(v, Jv) = \omega_0(v, Jv) - d\zeta(v, Jv) >0$. Define $g_0'(u, v) := \omega_0'(u, Jv)$ which is non-degenerate on $N$. Then, $g_0'(Jv, Jv) = -\omega_0'(Jv,v) = \omega_0'(v, Jv) = g_0'(v,v)$. The polarization identity yields 
	 	$$\omega_0(Ju, Jv) - d\zeta(Ju, Jv) = \omega_0'(Ju, Jv) = \omega_0'(u, v) = \omega_0(u, v) - d\zeta(u,v).$$ 
	 	We conclude that $d\zeta(Ju, Jv) = d\zeta(u,v)$. Since $\zeta$ is basic, we get that $Jd\zeta = d\zeta$ trivially on $T\F$.

	 	To show that $g'$ is Riemannian it only remains to show that $g'$ is positive-definite. We have that $g'(U, U) = g'(V, V) = 1 $. Independently of the structure considered, we always have the vector bundle decomposition $TM = T\F \otimes N$. Thus, it suffices to verify that $g'$ is positive-definite on $N$. Let $v\in N$ and calculate
	 	\begin{align*}
	 		g'(v,v) & =  \omega'(v, J'v) = \omega'(v,Jv)\\
	 		 & =  \omega_0'(v, Jv) + \theta\wedge\theta'^c(v,Jv) = \omega_0'(v,Jv)\\
	 		 & =  \omega_0(v,Jv) - d\zeta(v,Jv) > 0.
	 	\end{align*}
	 	Therefore, $(J', g', \theta) \in \mathcal{V}(J, g, \theta)$. 
	 	
	 	On the other hand, given $(J', g', \theta) \in \mathcal{V}(J, g, \theta)$ define $\zeta := \theta'^c-\theta^c$. We obtain that $\omega_0' = - d\theta'^c =  -d\theta'^c - d\zeta = \omega_0 - d\zeta$, which is positive-definite on $N'$ w.r.t. $J'$. Consequently, $\omega_0'$ is also positive-definite on $N$, since for any non-zero vector $v \in N$, we have the decomposition $v = v' + aU +bV$ with $v' \neq 0$ and $\omega_0'(v, J'v) = \omega_0'(v', J'v') > 0$. Hence, $\zeta \in C(M, \omega_0)$.
	 
	 \end{proof}
	 
	 Given $(J', g', \theta) \in  \mathcal{V}(J, g, \theta)$ write $\omega_1 := \omega_0'$. Consider $\zeta := \theta'^c - \theta^c$. Taking $\zeta_t := t\zeta$ for $t \in [0,1]$ yields $\omega_0 - td\theta = \omega_0 - t\omega_0 + t\omega_1 = t\omega_1 + (1-t)\omega_0$, which is positive-definite w.r.t. $J$ on $N$. By the above proposition we obtain a $1$-parameter family of Vaisman structures $(J_t, g_t, \theta)$ where $J_t$ defines a smooth $1$-parameter family of complex structures. In particular, $J_1$ and $J_0$ are in the same connected component of $\mathcal{C}$. We can apply Rollenske's result on \cite[Theorem 4.3]{Rollenske} to obtain the following.

	 \begin{Theo}\label{LeftInvDef}
	 	\mbox{}\
	 	\begin{enumerate}[label=(\arabic*)]
	 		\item Every complex structure $J$ given by a Vaisman deformation of $(J_0, g, \theta)$ is a deformation in the large of $J_0$.
	 		\item Let $M$ be a Kodaira-Thurston manifold with a Vaisman structure $(J, \omega, \theta)$ with left-invariant $J$. Then, the complex structure of every Vaisman deformation of $(J, \omega, \theta)$ is left-invariant. 
	 	\end{enumerate}
	 \end{Theo}

	\section{Vaisman Manifolds with Basic Harmonic 1-Forms of Constant Length}
	
	For a Riemannian manifold $(M, g)$ denote by $\hat{g} : TM \to T^*M$ the canonical isomorphism. Let $(M, \F, g)$ be a compact connected Riemannian foliated manifold. Denote by $\hat{g} : TM \to T^*M$ the canonical isomorphism. Set $\mathfrak{h}_B := \hat{g}^{-1}(\HH_B^1(M))$ the space of basic harmonic fields. Then, $\h_B$ is a finite dimensional vector space with $\dim_\R \h_B = b_1(\F)$. 
	
	Let $\alpha$ be a basic $1$-form. Then $A := \hat{g}^{-1}(\alpha) \in \Gamma N$, since $g(A, X) = \alpha(X) = 0$ for any $X \in T\F$. This implies that $\alpha = g(A,\bigcdot) = g_0(A, \bigcdot)$. 
	
	\begin{Claim}
		$A$ is a foliated vector field.
	\end{Claim}
	Indeed, for all $X \in T\F$ and $Y\in TM$ 
	$$0 = (\Lie_X \alpha)(Y) = X(\alpha(Y)) - \alpha([X,Y]).$$
	Hence, 
	\begin{align*}
		g(A,[X,Y]) & =  \alpha([X,Y]) = X(\alpha(Y)) = Xg_0(A,Y)\\
		& = (\Lie_X g_0)(A,Y) + g_0([X,A],Y) + g_0(A,[X,Y]).
	\end{align*}
	Therefore, $g_0([A,X],Y) = (\Lie_X g_0) (A, Y) = 0$ since $g_Q$ is holonomy invariant. Thus, $[A,X] \in T\F$.

	\begin{Def}
		Let $(M,\F, g)$ be a Riemannian foliated manifold. We say that a basic form $\alpha$ has constant length if $|\alpha|_g$ is a constant map, and that $(M, \F, g)$ has {\em basic harmonic} $1${\em-forms of constant length} if \textbf{all} forms $\alpha \in \HH_B^1(M)$ have constant length.
	\end{Def}
	
	\begin{Lemma}\label{PointwiseLILemma}
		Suppose $(M,\F, g)$ is a compact connected Riemannian foliated manifold with basic harmonic $1$-forms having constant length. Then, for any basic forms $\alpha_1, \alpha_2 \in \HH^1(M)$, we have that $g(\alpha_1, \alpha_2) = g(A_1, A_2)$ is a constant map on $M$.
		
		In particular, a basic harmonic vector field $A\neq 0$ is nowhere zero and, moreover, given an orthonormal basis $\{\alpha_1, \dots, \alpha_{b_1(\F)}\}$ with respect to the inner product defined on $\HH_B^1(M)$, the dual basis  $\{A_1, \dots, A_r\}$ defines a pointwise linearly independent orthonormal set of vector fields. 
	\end{Lemma}
	\begin{proof}
		This follows from the usual polarization identity.
	\end{proof}
	
	\begin{Rem}\label{Rem:CohLim}
		In the lemma above, any orthonormal basis of $\HH_B^1(M)$ gives rise to a basis of $\h_B$ which corresponds to a set of pointwise linearly independent global sections of the normal bundle $N$. This sets a cohomology constraint on $(M, \F)$, namely $b_1(\F) \leq \codim \F$.   
	
	\end{Rem}

	Now, consider $(M, J, \omega, \theta)$ a Vaisman manifold and take $X, Y \in \Gamma N$ local foliated vector fields. Recalling that $U = \theta^{\#}$ and $V = JU = \hat{g}^{-1}(\theta ^c)$, by the above discussion, $g(Z, X) = 0$ for any $Z \in T\Sigma$. We calculate:
	\begin{align*}
			0 & =  d\theta(X, Y) = X\theta(Y) - Y\theta(X) - \theta([X,Y])\\
			 & =  Xg(U, Y) - Yg(U,X) - \theta([X,Y])\\
			 & =  - \theta([X, Y]).
	\end{align*}
	Therefore, $g(U, [X,Y]) = \theta([X,Y]) = 0$. Once more,
	\begin{align*}
			-\omega_0(X, Y) & =  d\theta^c(X,Y)\\ 
			& =  X\theta^c(Y) - Y\theta^c(X) - \theta^c([X,Y])\\ 
			& =  - \theta^c([X,Y]) = - g(V, [X, Y]).
	\end{align*}
	A similar calculation shows that
	\begin{itemize}
		\item $g(U, [U, X]) = 0$.
		\item $g(V, [U, X]) = 0$.
		\item $g(U, [V, X]) = 0$.
		\item $g(V, [V, X]) = 0$.

	\end{itemize}	
	If we consider $\mathfrak{v} := \Span_\R \{U, V\}$, the real Lie algebra generated by $U$ and $V$, the above calculations show that $\g := \mathfrak{v} \oplus \h_B$ has a bracket operation satisfying
	\begin{align*}
		[X, Y] & = \omega_0(X,Y) V + [X, Y]^\perp\\
		[X, U] & = [X,V] = [U,V] = 0
	\end{align*}
	Now, $[X, Y]^\perp$ might not be an element of $\h_B$ and $\omega_0 (X,Y)$ can be non-constant, so the usual Lie bracket of vector fields can fail to define a Lie bracket on $\g$ in general. By imposing an additional constraint on the dimension of $\h_B$ and requiring the foliation to have basic harmonic $1$-forms of constant length, this problem vanishes as shown below. Before the next result, we define an almost complex structure on $\g$. On $\mathfrak{v}$, we have that $JU = V$. It remains to define a $J$ on $\h_B$. Recall that $\HH_B^1(M)_\C = \HH_B^{1,0}\oplus\HH_B^{0,1}$. Therefore, given a basis $\alpha_1, \dots, \alpha_k, \overline{\alpha_1}, \dots, \overline{\alpha_k}$ of $\HH_B^1(M)_\C$, Remark \ref{Rem:RelationOfCandRHarmForms} implies that $\alpha_j = \beta_j -iJ\beta_j $ for $\beta_j \in \HH_B^1(M)$. Thus, $\beta_1, \dots, \beta_k, J\beta_1, \dots, J\beta_k$ form a basis of $\HH_B^1(M)$. Hence, $J$ is defines an almost complex structure on $\HH_B^1(M)$ by linearity. Remark \ref{Rem:RelationOfCandRHarmForms} implies once more that $J$ descends to an almost complex structure on $\h_B$.

	\begin{Prop}\label{LieAlgStruc}
		Let $(M^{2n+2}, J, \omega, \theta)$ be a Vaisman manifold with basic harmonic $1$-forms of constant length. Assume that $b_1(M) = 2n+1$. Then, $\g$ is a Lie algebra with a left-invariant complex structure $J$ and $\omega_0$ is a left-invariant $2$-form. Moreover, $(\g,J)$ is isomorphic to $(\R\times \h_{2n +1}, J_0)$, where $J_0$ is the standard complex structure on $\R\times \h_{2n +1}$.
	\end{Prop}
	\begin{proof}
		As seen before, since basic harmonic $1$-forms have constant length, $\omega_0 (X,Y)$ is a constant for all $X, Y \in \h_B$. By hypothesis $\dim \h_B = b_1(\Sigma) = \dim_\R M - 2 = \rk T\F^{\perp}$, which implies that a basis of $\h_B$ trivializes $N$. Now, for any other $A \in \h_B$
		$$g([X,Y], A) = \alpha([X, Y]) = X\alpha(Y) - Y\alpha(X) - d\alpha(X, Y) = 0.$$
		Hence, $[X,Y] = \omega_0(X,Y) V$ for $X,Y \in \h_B$. We conclude that $[\bigcdot,\bigcdot]$ defines a Lie bracket on $\g$ with the same relations of the bracket on $\R\times \h_{2n + 1}$. The almost complex structure $J$ on $\g$ defined above is a complex structure on $\g$ since the bracket $[\bigcdot,\bigcdot]$ is the usual Lie derivative on $M$. By completing $U, V$ to a $J$-invariant orthonormal basis of $\g$, we obtain a isomorphism from $(\g, J)$ to $(\R\times \h_{2n +1}, J_0)$.
	\end{proof}
	
	As mentioned in the introduction, we can finally prove the Vaisman version of the result in \cite{NagyLengthofHarm}.
	\begin{Theo}\label{Constant Length}
		Let $(M^{2n+2}, J, \omega, \theta)$ be a Vaisman manifold with basic harmonic $1$-forms having constant length. Assume that $b_1(M) = 2n+1$. Then, $M$ is diffeomorphic to a Kodaira-Thurston manifold where $J$ is the standard left-invariant complex structure under this diffeomorphism.
	\end{Theo}
	\begin{proof}
		By the proposition above and \cite[Corollary 3 on Page 113]{OnishchikLieGroups}, it follows that $G := \R \times H_{2n+1}$ defines a Lie group action on $M$ through the flows of a basis of $\g$. Fixing a point $p \in M$ we can define $\phi_p : G \to M$ by $\phi_p (g) := g\bigcdot p$. As before, $d(\phi_p)_e (X) = \vec{X}_p = X_p$ for any $X \in \g$. Since the vector fields are non-zero everywhere and they form a pointwise basis for $T_pM$, $ d(\phi_p)_e$ is an isomorphism for any $p$, hence $\phi_p$ is a local diffeomorphism. Since $G$ and $M$ are connected manifolds with the same dimension, we conclude that $\phi_p (G) = M$. This shows that the action is transitive, so $M$ is diffeomorphic to a Kodaira-Thurston manifold and $J$ is left-invariant.

	\end{proof}

	\begin{Rem}\label{Rem:CounterExample}
		
		We discuss the scope of the above theorem. What happens if $b_1(M)$ is arbitrary? First notice that $0 < b_1(M) \leq 2n +1$ for any Vaisman manifold $M$ by Remark \ref{Rem:CohLim} and Proposition \ref{Prop:Kashiwada}. In \cite{AndradaOriglia} the authors constructed families of solvmanifolds with left-invariant Vaisman structures with varying first Betti number. We consider one of these families and show that they have basic harmonic $1$-forms of constant length. 
		
		Consider the basis $\{Z, X_1, Y_1, \dots, X_n, Y_n\}$ satisfying $[X_i, Y_i] = Z$, with all other brackets trivial. Let $a_1, \dots, a_n \in \R$ and define
		\[
		D:=\begin{pmatrix}
			0 &  &  &  &  & \\
			&  0& -a_1 &  &  &  \\
			&  a_1&  0&   &  &  \\
			&  &  & \ddots &  &   \\
			&  &  &  & 0 & -a_n \\
			&  &  &  & a_n & 0 \\
		\end{pmatrix}
		\]
		
		Let $\g = \g_{a_1, \dots, a_n} := \R A \ltimes_D \h(1,n)$, where $A$ is just a symbol representing a generator. These are called the {\em oscillator Lie algebras}. One can show that if $\exists c \in \R\setminus \{0\}$ such that $(a_1, \dots, a_n) = c (b_1, \dots, b_n)$, then $\g_{a_1, \dots, a_n}$ and $\g_{b_1, \dots, b_n}$ are isomorphic. All these Lie algebras admits a Vaisman structure in the following manner. Define a complex structure $J$ by the relations $JA := Z, JX_i := Y_i$, and consider a metric $g$ given by declaring $\{A, Z, X_1, Y_1, \dots, X_n, Y_n\}$ an orthonormal basis of $\g$. Then, $g$ is a Hermitian metric by construction. Consider $\theta := A^*$ the dual of $A$ and $\omega$ the fundamental form. By \cite[Theorem 3.10]{AndradaOriglia} $(J, \omega, \theta)$ defines a Vaisman structure on $\g$. 
		
		We consider now the associated Lie group $G = G_{a_1, \dots, a_n} := \R \ltimes_{\varphi} H(1,n)$, where we are considering $H(1,n)$ being described as $\R^{2n+1}$ and $\varphi(t) := e^{tD}$. Therefore,
		\[
		\varphi(t) = \begin{pmatrix}
			1 &  &  &  &  & \\
			&  \cos ta_1 & -\sin ta_1 &  &  &  \\
			&  \sin ta_1& \cos ta_1&   &  &  \\
			&  &  & \ddots &  &   \\
			&  &  &  & \cos ta_n & -\sin ta_n \\
			&  &  &  & \sin ta_n & \cos ta_n \\
		\end{pmatrix}.
		\]
		
		In \cite{AndradaOriglia}, the authors takes $a_i \in \Z$ for all $i$. In this way, they consider a fixed lattice $\Gamma_k := \frac{1}{2k} \Z \times \dots \times \Z$ for $H(1,n)$ and construct three families of lattices in $G$ by
		\begin{equation*}
			\Lambda_{k, \frac{\pi}{2}} := \frac{\pi}{2}\Z \ltimes_{\varphi} \Gamma_k, \qquad  \Lambda_{k, \pi} := \pi\Z \ltimes_{\varphi} \Gamma_k, \qquad \Lambda_{k, 2\pi} := 2\pi\Z \ltimes_{\varphi} \Gamma_k.
		\end{equation*}
		If $D = 0$, then one recovers $\R\times H(1,n)$ with its standard Vaisman structure. When one consider the expression of the left-invariant structure $(J, \omega, \theta)$ over the Lie group $G$, by a direct calculation it can be shown that it coincides with the expression for the metric over the nilpotent group $\R\times H_{2n+1}$. In \cite{Gomes} we showed how the solvmanifold $M_1 := \Lambda_{k, \pi} \backslash G$ is a quotient of the nilmanifold $M_0 := \Lambda_{k, 2\pi} \backslash G$ by a $\Z_2$-action. Therefore, the Vaisman structure above descends to the \textbf{standard} Vaisman structure on $M_0$ which is invariant over the $\Z_2$-action. This implies that $\Z_2$ acts by isometries. Since under the standard structure the harmonic forms on $M_0$ are left-invariant, they have constant length. By the isometric action we obtain that $\HH^1(M_0)^{\Z_2} \cong \HH^1(M_1)$. Thus, $M_1$ also has basic harmonic $1$-forms of constant length. For $M_2 := \Lambda_{k, \frac{\pi}{2}} \backslash G$ a similar argument applies.
		
		In \cite{Gomes} it was shown that every Vaisman solvmanifold is a finite quotient of a Kodaira-Thurston manifold. With these results in mind, a naive expectation would be that, for arbitrary $b_1(M)$, $M$ must be diffeomorphic to a solvmanifold. This already fails for $b_1(M) = 1$. Indeed, the classical Hopf manifold $M \cong S^1 \times S^{2n+1}$ is a Vaisman manifold with $b_1(M) = 1$. It has basic harmonic $1$-forms of constant length trivially and it cannot be a solvmanifold because it is not aspherical. Since for $b_1(M) = 1$ any Vaisman manifold has basic harmonic $1$-forms of constant length trivially, perhaps this is a degenerate case and asking for $b_1(M) > 1$ our expectation should hold. This is also does not work and we construct an example of with $ 2n +1> b_1(M) > 1$ which is not aspherical in the following. Let $A$ be an Abelian variety of $\dim_\C = p$. Consider the projective K\"ahler manifold $A \times \C P^q$ with the product K\"ahler metric, where $q \geq 1$.  Using the construction method of the Boothby-Wang fibration, we obtain a Sasakian manifold $S^1 \xlongrightarrow{} S \xlongrightarrow{\pi} A\times \C P^q$ with $\dim S = 2 p + 2q +1$. Taking  the trivial Vaisman extension $M := S^1 \times S$ we see that $b_1(M) = 2p + 1 < 2p + 2q + 1$ and $M$ has all basic $1$-forms of constant length.

	\end{Rem}

	\section{The Main Theorem}

	We start with a Lemma.
	
	\begin{Lemma}\label{AlbMaxRankLemma}
		Let $(M, J, \omega, \theta)$ be a Vaisman manifold. Suppose there exists a basis of $\HH_B^1(M)_\C$ which is pointwise linearly independent at all points of $M$. Then, the Albanese map $\alpha$ has maximal rank everywhere. In particular, if $M$ has basic harmonic $1$-forms of constant length, the Albanese map has maximal rank everywhere. 
	\end{Lemma}
	\begin{proof}
			Set $2n := b_1(\Sigma)$. By hypothesis, let $\omega_1, \dots, \omega_2n$ be a $\C$-basis of $\HH_B^1(M)_\C$ which is $\C$-linearly independent at every point $p\in M$. In particular, they are non-zero at every point. Consider the decomposition 
			$$\omega_j = \frac{1}{2}(\omega_j - iJ\omega_j) + \frac{1}{2}(\omega_j + iJ\omega_j).$$
			Then $\alpha_j := \frac{1}{2}(\omega_j - iJ\omega_j) \in \HH_B^{1,0}$ form a basis of $\HH_B^{1,0}$ which is pointwise linearly independent. By Remark \ref{Rem:RelationOfCandRHarmForms}, $\alpha_j = \beta_j - i J\beta_j$ where $\beta_j \in \HH_B^1(M)$. Therefore, $\beta_1, \dots, \beta_n,$ $J\beta_1, \dots, J\beta_n$ is a basis for $\HH_B^1(M)$ which is $\R$-linearly independent at every point. Consider $B_1, \dots, B_n, JB_1, \dots, JB_n$ the induced basis on $\h_B (M)$. Since, for any point $p \in M$ we have that $d\alpha_p [v] = \iota_v|_{\Omega^1}$, we obtain in particular that $d \alpha_p[B_j] = \iota_{B_j|_p}$. Now, $\iota_{B_j|_p}(\alpha_l) = \beta_l(B_j|_p) = \delta_{jl}$. Similarly,  $\iota_{JB_j|_p}(\alpha_l) = - \delta_{jl}$, so $d\alpha_p$ has maximal rank for any $p$.
		
		When $M$ has basic harmonic $1$-forms of constant length, we can choose an orthonormal basis $\alpha_1, \dots, \alpha_n$ of $\HH_B^{1}(M)_\C$, and this is automatically orthonormal at every point $p \in M$ by Lemma \ref{PointwiseLILemma}.

	\end{proof}

	\begin{Theo}\label{MaximalRankTheo}
		Let $M^{2n+2}$ be a Vaisman manifold with $b_1(M) = 2n+1$ and assume that its Albanese map $\alpha$ has maximal rank everywhere. Then, $\Sigma$ is regular and the Boothby-Wang fibration of $(M, \Sigma)$ is given by the Albanese map as shown below. 
		$$ T^2 \longrightarrow M \xlongrightarrow{\alpha} \Alb (M).$$	 
	\end{Theo}
	
	\begin{proof}
		\mbox{}\\
		\textbf{Step 1:} The structure is quasi-regular.
		
		By hypothesis, $\alpha$ has maximal rank everywhere. Hence, $\alpha^{-1}(q)$ is a compact $2$-dimensional embedded submanifold of $M$ for all $q\in \Alb(M)$. Since each fiber is compact, they have a finite number of connected components. Since $\alpha$ is constant over the leaves of the foliation, every leaf is contained into a fiber of $\alpha$. Since every leaf is $2$-dimensional and connected, it is a connected component of a fiber. This can be seen as follows: Let $S \subset M$ be a fiber of $\alpha$. For every $p \in S$, consider $L_p \Sigma$ the leaf containing the point $p$. Hence, $S = \bigcup_{p\in S}L_p$. We can consider the distribution $D := T\Sigma|_S$ on $S$. Since every leaf in $S$ has the same dimension of $S$, $D = TS$. Thus, each connected component of $S$ is a integral manifold of $D$. Since a leaf of $\Sigma$ inside $S$ is also a maximal integral manifold of $D$, we get that each connected component of $S$ is a leaf of $\Sigma$. In particular, every leaf of $\Sigma$ is compact.

		\textbf{Step 2:} The structure is regular.
		
		Since $\Sigma$ is quasi-regular, Proposition \ref{Prop:QuasiRefVaismanM} shows that $\Sigma = \{T^2\bigcdot p \}_{p \in M}$ for a complex torus $T^2$ defined by the flows of $U$ and $V$. We calculate for any $p \in M$ and $g\in T^2$
		$$\alpha \circ L_g(p) = \left[\int_{p_0}^{L_g(p)}\right] = \left[\int_{p_0}^{p}\right] + \left[\int_{p}^{L_g(p)}\right].$$
		Since $T^2$ defines the leaves of $\Sigma$, $\left[\int_{p}^{L_g(p)}\right] = 0$, hence $\alpha \circ L_g(p) = \alpha(p)$. In particular, $d(\alpha_p\circ L_g)_p = d\alpha_p$. Now, suppose $g \in T^2_p$ and consider $d(L_g)_p : T_pM \to T_pM$. Since $T^2$ acts by isometries on $M$, it preserves the normal bundle $N$, which means that $d(L_g)_p|_{N_p} : N_p \to N_p$ is an isomorphism. The Albanese map has maximal rank everywhere, hence $\ker d\alpha = T\Sigma$. Thus, $d\alpha|_N$ is an isomorphism. Since $p = L_g(p)$, this implies that $d(L_g)_p|_{N_p} = \Id_{N_p}$. Recall that $L_g$ preserves action fields $\tilde{X}_p := \frac{d}{dt}|_{t = 0} \exp tX \bigcdot p$. Since they define a basis for $T_p\Sigma$, $d(L_g)_p|_{T_p\Sigma} = \Id_{T_p\Sigma}$. We obtain that $d(L_g)_p = \Id_{T_pM}$. Since $M$ is connected and $L_g$ is a isometry for any $g \in T^2$, we obtained that $L_g = \Id_M$ for all $g\in T^2_p$. We conclude that $T^2$ acts freely and by biholomorphic isometries on $M$, therefore it defines a complex Hermitian manifold $\X = M \slash T^2$.

		\textbf{Step 3:} The Boothby-Wang fibration of $(M, \Sigma)$ is given by the Albanese map.
		
		Let $\pi : M \to \X$ be the natural quotient map. Thus, $\pi$ is a Riemannian holomorphic submersion. Since $\alpha$ is constant along the fibers of $\pi$, we obtain that there exists a unique holomorphic map $\beta : \X \to \Alb (M)$ such that the diagram below commute.
		\[
		\begin{tikzcd}
			M \arrow[r, "\alpha"] \arrow[d, "\pi"'] & \Alb(M) \\
			\X \arrow[ru, "\beta"', dashed]         &        
		\end{tikzcd}
		\]
		Recall that $\omega = \omega_0 + \theta\wedge\theta^c$. Since $\theta\wedge\theta^c|_N = 0$ we obtain that $\omega$ projects to $\omega_1$ on $\X$ in such a way that $\pi^*\omega_1 = \omega_0$. In particular $d\omega_1 = 0$, meaning $\omega_1$ is K\"ahler. By the above diagram, $\beta$ is a holomorphic covering map between K\"ahler manifolds. In fact, $\X$ is a compact connected aspherical K\"ahler manifold. By \cite{Hasegawa, BensonGordon}, $\X$ is biholomorphic to a complex torus $T^{2n}$. 
		
		Using the universal property of the Albanese map, there exists a unique holomorphic map $\delta : \Alb(M) \to \X$ such that the following diagram commutes.
		\[
		\begin{tikzcd}
			M \arrow[r, "\alpha"] \arrow[d, "\pi"'] & \Alb(M) \arrow[ld, "\delta"', dashed] \\
			    \X      &        
		\end{tikzcd}
		\]
		Thus, 
		$$\alpha = \beta \circ \pi \implies \pi = \delta \circ \alpha = \delta \circ \beta \circ \pi \implies \delta \circ \beta = \Id_\X.$$ On the other hand,		
		$$\pi = \delta \circ \alpha \implies \alpha = \beta \circ \pi = \beta \circ \delta \circ \alpha \implies \beta \circ \delta = \Id_{\Alb(M)}.$$
		Therefore, $\beta$ is a biholomorphism between $\Alb(M)$ and $\X$. The action of $T^2$ is holomorphic and free, hence the map $\phi_p : T^2 \to M$ defined by $\phi_p(g) := g\bigcdot p$ restricts to a biholomorphism $\phi_p : T^2 \to T^2 \bigcdot p$, meaning every leaf of $\Sigma$ is biholomorphic to $T^2$. Since $M$ is compact, $\pi$ is a proper map. Since $\X$ is connected and $\pi$ is a proper submersion, Ehresmann's Fiber Bundle Theorem implies that $M$ is a smooth fiber bundle 
		$$ T^2 \longrightarrow M \longrightarrow T^{2n} \cong \Alb(M).$$

	\end{proof}

	\begin{Cor}\label{MaxRankCoro}
		Let $(M^{2n+2}, J, \omega, \theta)$ be a Vaisman manifold with basic harmonic $1$-forms of constant length. Assume that $b_1(M) = 2n+1$. Then, the following are true:
		\begin{enumerate}[label=(\arabic*)]
			\item $M$ is diffeomorphic to a Kodaira-Thurston manifold where $J$ is the standard left-invariant complex structure.
			\item The characteristic foliation $\Sigma$ is regular and the Boothby-Wang fibration of $(M, \Sigma)$ is given by the Albanese map.
			\item $\omega_0$ is a flat left-invariant K\"aler metric on $\Alb (M)$.
		\end{enumerate}
	\end{Cor}
	\begin{proof}
		Items $(1)$ and $(2)$ follow from Theorem \ref{Constant Length} and Theorem \ref{MaximalRankTheo}. Item $(3)$ follows from Proposition \ref{LieAlgStruc} together with the fact that all left-invariant metrics on a torus are flat.
	\end{proof}
	
	Using a version of the Bochner method for foliations, one obtains the following result.
	\begin{Prop}[\cite{HabibModifiedDiff}, \cite{RoschigBochnerTechnique}]
		Let $(M, \F, g)$ be a Riemannian foliation with $\Ric^T = 0$. Then, all basic harmonic fields are transverse parallel. In particular, they all have constant length.
	\end{Prop}
	
	We recall the transverse Calabi-Yau Theorem.
	
	\begin{Theo}[\cite{ElKacimi}]
		Let $(M, \F, \omega_0)$ be a transversely K\"ahler foliated manifold. Then, for any $\rho \in c_{1, B}(M)$ there exists a unique transverse K\"ahler form $\omega_1 \in [\omega_0]_B$ such that $\rho$ is the transverse Ricci form of $\omega_1$.
	\end{Theo}
	
	In this case, we obtain that $\omega_1 = \omega_0 - d\zeta$ with $\zeta$ being a basic $1$-form. From our discussion in section \ref{DefOfVaismSection}, since $\omega_1$ is a transverse K\"ahler form we have that $\zeta \in C(M, J, \omega_0)$. Then, we can consider the associated Vaisman deformation $(J', g', \theta) \in \mathcal{V}(J, g, \theta)$ as defined in the Proposition \ref{VaismanDefProp}. Thus, from the transverse Calabi-Yau Theorem we can obtain a new Vaisman structure $(J', \omega', \theta)$ with the same Lee form $\theta$, $J'$ a deformation on the large of $J$ and with transverse K\"ahler form $\omega_1$ with transverse Ricci form $\rho$. We can finally prove the main theorem of this work.
	
	\begin{Theo}\label{MainTheorem}
		Let $(M^{2n+2}, J, \omega, \theta)$ be a Vaisman manifold with $b_1(M) = 2n+1$ and $c_{1, B}(M) = 0$. Then, $M$ is diffeomorphic to a Kodaira-Thurston manifold and $J$ is left-invariant. Moreover, the characteristic foliation $\Sigma$ is regular and the Boothby-Wang fibration of $(M, \Sigma)$ is given by the Albanese map.

	\end{Theo}
	\begin{proof}
		By the above discussion, $M$ admits a Vaisman deformation $(J', g', \theta)$ with vanishing transverse Ricci form. In particular, the metric given by $g'$ is transverse Ricci-flat, so all basic harmonic forms have constant length by the proposition above. Furthermore, Theorem \ref{Constant Length} shows that $M$ is diffeomorphic to a Kodaira-Thurston manifold with $J'$ being left-invariant. By Theorem \ref{LeftInvDef}, $J$ is also left-invariant. 
		
		Denote by $J^*$ the left-invariant complex structure on $\R\times H_{2n +1}$ such that the projection map $(\R\times H_{2n +1}, J^*) \to (M, J)$ is holomorphic. By \cite[Proposition 3.3]{Rollenske}, $J$ is a nilpotent complex structure on the Lie algebra $\mathfrak{g} := \R\times\mathfrak{h}_{2n+1}$. Therefore, by \cite{CorderoEtAl} we have access to a Dolbeault cohomology version of Nomizu's theorem. Namely, there exists a natural isomorphism $H_{\overline{\partial}_{J^*}}^{*, *}(\g\otimes \C) \cong H_{\overline{\partial}_{J}}^{*, *}(M)$. In particular, $\Omega^1(M) = H^{1, 0}(M)$ is composed of left-invariant forms. Given a basis $\alpha^1, \dots, \alpha^n$ of $\Omega^1(M)$, since each form is left-invariant, they are linearly independent at each point of $M$. Since holomorphic $1$-forms are basic harmonic, the result follows from Lemma \ref{AlbMaxRankLemma} and Theorem \ref{MaximalRankTheo}.
	\end{proof}
	\begin{Rem}\label{Rem:MainTheo}
		From the hypothesis $c_{1,B}(M) = 0$, we can always find a Vaisman structure with basic harmonic $1$-forms of constant length for $M$. Since the Betti numbers are topological invariants, the discussion of \ref{Rem:CounterExample} applies here. In conclusion, $b_1(M) = 2n +1$ is the highest value we can impose for a Vaisman manifold with $c_{1,B}(M) = 0$, and for $0 < b_1(M) < 2n+1$ we have no conclusive characterization a priori.
	\end{Rem}

	\section{Transversely Geometrically Formal Vaisman Manifolds}
	
	In \cite{Kotschick} Kotschick introduced the notion of geometrically formality of Riemannian metrics. 
	
	\begin{Def}
		Let $(M, g)$ be a compact connected oriented manifold. Then the metric $g$ is called {\em geometrically formal} if the wedge product of harmonic forms is a harmonic form.
	\end{Def}
	
	They seem to share a lot of properties with flat Riemannian metrics and in fact every flat Riemannian metric is geometrically formal. 
	
	\begin{Prop}[Theorem 10 in \cite{sferruzza2025hermitiangeometricallyformalmanifolds}]
		Let $(M, g)$ be a compact connected oriented manifold with nonnegative curvature operator. Then $g$ is geometrically formal. In particular, every flat metric on $M$ is geometrically formal. 
	\end{Prop}
	\begin{proof}
		The proof is essentially the same as \cite[Theorem 10]{sferruzza2025hermitiangeometricallyformalmanifolds}. The authors assume $(M, g)$ to be K\"ahler, however a careful look at their proof reveals that it is not necessary. The proof boils down to use a Gallot-Meyer theorem which tells us that if the curvature operator is nonnegative, then all harmonic forms are parallel. By the Leibniz rule, the wedge product of parallel forms are parallel, and by the description of the differential and co-differential in terms of the Levi-Civita connection, every parallel form is harmonic, hence $g$ is geometrically formal.
	\end{proof}
	
	It is no surprise that this type of formality has been generalized and gained interest recently in the context of Riemannian foliation \cite{HabibTGF}.
	
	\begin{Def}
		Let $(M, \mathcal{F}, g)$ be a Riemannian foliated manifold. The metric $g$ is called {\em transversely geometrically formal} (TGF) if the wedge product of basic harmonic forms is a basic harmonic form.
	\end{Def}
	\begin{Rem}
		This definition requires the foliation to be taut for it to be well-defined. See \cite{HabibTGF} for a more general definition.
	\end{Rem}
	
	An important property that these metrics have is that every basic harmonic form have constant length (\cite[Lemma 3.4]{HabibTGF}). Therefore, every TGF Riemannian metric on $(M, \F)$ have basic harmonic $1$-forms of constant length. A priori, TGF seems so be a much restrictive condition on $g$ than having basic harmonic $1$-forms of constant length, however for Vaisman manifolds with first betti number $b_1 = \dim_\R M - 1$ they are equivalent precisely because of our classification in the previous sections. 
	 
	\begin{Prop}
		Let $(M, g, \F)$ be a taut Riemannian foliated manifold with bundle-like metric $g$. Assume further that $\F$ is regular and let $\X := M\slash \F$ be the quotient manifold. Then, $(M, \F, g)$ is TGF if and only if $(\X, g_0)$ is geometrically formal.
	\end{Prop}
	\begin{proof}
		This is a direct consequence of the fact that a form $\alpha$ in $\X$ is harmonic if and only if $\pi^*\alpha$ is basic harmonic.
	\end{proof}

	\begin{Prop}
		Let $(M^{2n+2}, J, \omega, \theta)$ be a Vaisman manifold with $b_1(M) = 2n+1$. Then, $(M, \Sigma, g)$ has basic $1$-forms of constant length if and only if the Vaisman metric $g$ is TGF.
	\end{Prop}
	\begin{proof}
		If $M$ has basic $1$-forms of constant length, Corollary \ref{MaxRankCoro} implies that the induced metric on $\Alb(M)$ is flat. Theorem 10 in \cite{sferruzza2025hermitiangeometricallyformalmanifolds} (or the proposition above) shows that $\Alb(M)$ is geometrically formal, which implies that $g$ is TGF.
		
		The other direction follows directly from \cite[Lemma 3.4]{HabibTGF}. 
	\end{proof}

	\section{The Mapping Torus and Vaisman Structures of LCK Rank 1}\label{TheMappingTorus}
	
	Let $(M, g)$ be a compact connected Riemannian manifold and let $\phi : M \to M$ be an isometry. Fix a real constant $a > 0$. Denote by $\Iso(M)$ the isometry group and $\Iso(M)_0$ the component of the identity. Define $\rho : \R \times M \to \R \times M $ by $\rho(t, p) := (t+a, \phi(p))$. Thus, $\rho$ induces a $\Z$ action on $\R \times M$ through iterations of $\rho$ on $M$. This action is proper, smooth and free, hence it defines a quotient manifold by 	
	$$M_{\phi, a} := \faktor{(\R \times M)}{\rho_{\phi, a}},$$
	called the {\em mapping torus} of $M$ induced by the pair $(\phi, a)$. By extending the metric $g$ to $\tilde{g} := g + dt^2$, $\rho$ is an isometry of $(\R \times M, \tilde{g})$, so $\tilde{g}$ descends to a metric $\overline{g}$ on $M_{\phi, a}$. 
	One can show that the mapping torus sits in a commutative diagram
	\[
	\begin{tikzcd}
		\R \times M \arrow[d, "p_a"'] \arrow[r, "p_2"] & \R \arrow[d, "q_a"] \\
		{M_{\phi, a}} \arrow[r, "\pi"']                 & S_a^1 := \frac{\R}{a\Z}   
	\end{tikzcd}
	\]
	where $p_a$ and $q_a$ are the quotient maps of the respective actions, $p_2$ is the projection on the second coordinate and $\pi([x, t]) := [t]$. This map is well-defined and since $dt$ on $\R$ is invariant under the action of $a\Z$, we obtain that $dt$ descends to a closed $1$-form $\theta_a$ on $S_a^1$. Furthermore, one can show that $M_{\phi, a}$ is compact and $\pi$ is a Riemannian submersion. In particular, $M_{\phi, a}$ is a fiber bundle over $S_a^1$ with $M$ fibers.
	
	\begin{Theo}[\cite{BazzoniOnTheStructureofCoKahler}]
		Let $(M, g)$ be a compact Riemannian manifold and let $\phi \in \Iso(M)$. Then, $(M_{\phi, a}, \phi)$ is isometric to $(S_a^1 \times M, \theta_a\otimes \theta_a + g)$, where $\theta$ is defines as above, if and only if $\phi \in \Iso_0(M)$.
	\end{Theo}
	
	Since $\Iso(M)$ is a compact finite dimensional Lie group, one can show that given any $\phi \in \Iso(M)$ there exists an $m > 0$ such that $\phi^m \in \Iso_0(M)$. Then, by taking $f := \phi^m$ and $b := ma$ we obtain a locally isometric finite covering $u : M_{f, b} \to M_{\phi, a}$ with $M_{f,b}$ isometric to $(\R\slash b\Z) \times M $ (see \cite[Remark 1]{Bazzoni2}).

	\begin{Theo}\label{Theo: theta is left-inv}
		Let $M = \Gamma \backslash H_{2n+1}\times \R$ be a Kodaira-Thurston manifold with $\pi : H_{2n+1}\times \R \to M$ being the quotient map. Denote by $t : \R \times H_{2n+1} \to \R$ the projection in the first component and let $(J, \omega, \theta)$ be a Vaisman structure on $M$ with LCK rank $1$. Then, $\pi^*\theta = dt$. In particular, $\theta$ is a left-invariant form.

	\end{Theo}
	\begin{proof}
		 Denote by $\pi : \tilde{M} \to M$ the universal covering of $M$ and pullback the metric $g$ to $\tilde{g} := \pi^*g$. By the structure theorem of Vaisman manifolds (\cite{OVStructure, OV1}), there exists a compact connected Sasakian manifold $S$ such that $M$ is isometric to a mapping torus $S_{\phi, a}$. This isometry arises by lifting the parallel Lee vector field $U$ to a parallel vector field $\tilde{U}$ on $\tilde{M}$ through $p$. The vector $\tilde{U}$ then yields a splitting $\tilde{M} \cong \R \times S$, where $\R$ is parametrized by the flow of $\tilde{U}$. The splitting induces the metric decomposition $\tilde{g} = g_S + ds^2$ with $ds = \pi^*\theta$. Thus, there exists a $m > 0$ and $b = ma$ such that $(S_{f, b}, \overline{g})$ is isometric to $(S_b^1 \times S, \theta_b\otimes \theta_b + g_S)$ by the above discussion. 
		 
		 Denote by $u : S_{f, b} \to S_{\phi, a}$ the induced locally isometric finite covering. Since $M$ has the structure of a Kodaira-Thurston manifold, the long exact sequence of homotopy groups shows that $S$ is a compact aspherical Sasakian manifold with nilpotent fundamental group $\Lambda := \pi_1(S)$. By \cite{Yudin2}, $\Lambda$ can be identified with a lattice in $H_{2n+1}$ in such a way that there exists a diffeomorphism $\psi : S \to \Lambda \backslash H_{2n+1}$. We can define a diffeomorphism $\Pi : \R \times H_{2n+1} \to S_b^1 \times S $ by $\Pi := q_b \times \psi$. By construction $\Pi^*\theta_b = dt$ and $u^*\theta = \theta_b$. Thus, $\Pi^*u^*\theta$ is precisely $dt$ on $\R \times H_{2n+1}$, which is left-invariant. Since $\Gamma':= b\Z \times \Lambda$ is a sublattice of $\Gamma$ by \cite[Lemma 3.1]{Gomes}, the following diagram commutes. 
		 \[
		 \begin{tikzcd}
		 	\R\times H_{2n+1} \arrow[rd, "\Pi"] \arrow[dd, "\pi"'] &                               \\
		 	& S_b^1\times S \arrow[ld, "u"] \\
		 	{M \cong S_{f, a}}                                     &                              
		 \end{tikzcd}
		 \]
		 Therefore, we obtain that $u \circ \Pi = \pi$, hence $ds = \pi^*\theta = \Pi^*u^*\theta = dt$, implying that $\theta$ is left-invariant.
	\end{proof}
	
	\begin{Theo}
		Let $(M^{2n+2}, J, \omega, \theta)$ be a Vaisman manifold with LCK rank $1$, $b_1(M) = 2n+1$ and $c_{1,B}(M) = 0$. Then, $M$ is a Kodaira-Thurston manifold and $(J, \omega, \theta)$ is a left-invariant Vaisman structure on $M$. In addition, the characteristic foliation $\Sigma$ is regular with the Albanese map being the Boothby-Wang fibration of $M$ over $\Alb (M)$.

	\end{Theo}
	\begin{proof}
	
		By Theorem \ref{MainTheorem} it only remains to show that $(\omega, \theta)$ is left-invariant. By the above theorem, $\theta$ is left-invariant. Since $M$ is Vaisman, the fundamental form can be written as $\omega = -dJ\theta + \theta\wedge J\theta$, which is an expression given by left-invariant forms.
	
	\end{proof}

	\printbibliography
\end{document}